\documentclass[11pt]{amsart}

\usepackage{enumerate,url,upref,amsrefs}

\theoremstyle{cupthm}
\newtheorem{theorem}{Theorem}[section]

\newtheorem{corollary}[theorem]{Corollary}
\newtheorem{lemma}[theorem]{Lemma}
\theoremstyle{cupdefn}

\theoremstyle{cuprem}

\numberwithin{equation}{section}

\newtheorem{question}[theorem]{Question}
\newtheorem{example}[theorem]{Example}

\newcommand{\HH}{\mathcal{H}}

\newcommand{\R}{\mathbb{R}}

\newcommand{\be}[1]{\begin{equation}\label{#1}}
\newcommand{\ee}{\end{equation}}

\begin{document}
\baselineskip5.6mm

\title{Lipschitz retraction of finite subsets of Hilbert spaces}

\author{Leonid V. Kovalev}
\address{215 Carnegie, Mathematics Department, Syracuse University, Syracuse, NY 13244-1150}
\email{lvkovale@syr.edu}
\thanks{Supported by the National Science Foundation grant DMS-1362453.}


\begin{abstract}
Finite subset spaces of a metric space \(X\) form a nested sequence under natural isometric embeddings \(X=X(1)\subset X(2)\subset\dots\). 
We prove that this sequence admits Lipschitz retractions \(X(n)\to X(n-1)\) when \(X\) is a Hilbert space. 
\end{abstract}

\maketitle

\section{Introduction}

Let $X$ be a metric space. For $n\ge 1$, $X(n)$ denotes the set of all nonempty subsets of $X$ with cardinality at most $n$. Equipped with the Hausdorff metric, $X(n)$ becomes a metric space which is sometimes called
a symmetric product or symmetric power of $X$. Following Tuffley~\cite{Tu1} we use the descriptive term 
\emph{finite subset space} for $X(n)$. This space is related to, but different from $X^n/S_n$, the space
of unordered $n$-tuples of points in $X$. 

One feature that distinguishes $X(n)$ from the Cartesian power $X^n$ and the quotient $X^n/S_n$ is the existence of canonical isometric embeddings $X(n)\subset X(n+1)$. The embeddings $X^n\subset X^{n+1}$ are not canonical: they depend on the choice of a base point in $X$. Furthermore, the geometry of embedding 
$X(n)\subset X(n+1)$ is far richer than the geometry of $X^n\subset X^{n+1}$. For example~\cite{Mo}, $S^1(1)\subset S^1(3)$ is a trefoil knot embedded in $S^3$ which bounds a M\"obuis band, namely $S^1(2)$. 
This example shows that in general the canonical embeddings $\iota\colon X(n)\to X(n+1)$ do not split: 
there need not be a continuous map $r\colon X(n+1)\to X(n)$ such that $r\circ \iota=\mathrm{id}$. 

On the other hand, there is a Lipschitz retraction of $\mathbb R(n+1)$ onto $\mathbb R(n)$ for every $n\ge 1$.
This observation, made in~\cite{Ko}, was used to show the bi-Lipschitz embeddability of $\mathbb R(n)$ into a Euclidean space. Our main result shows that such a Lipschitz retraction exists for all Hilbert spaces, either finite-dimensional or infinite-dimensional. In this context, there is no loss of generality in assuming the vector spaces are real. 

\begin{theorem}\label{thm1} 
Let $\HH$ be a Hilbert space.  Then for every $n\ge 2$ there exists a Lipschitz retraction $r_n\colon \HH(n )\to \HH(n-1)$.  
\end{theorem}

By Remark 4.5 in~\cite{Ko}, combining the case $X=\mathbb R^d$ of Theorem~\ref{thm1} with the results of~\cite{Ko} yields the following corollary.

\begin{corollary}\label{embedcor} For $d,n\ge 1$, the space $\mathbb R^d(n)$ is an absolute  Lipschitz retract. 
\end{corollary} 

Since the existence of Lipschitz retractions $r\colon X(n)\to X(n-1)$  is a bi-Lipschitz invariant of $X$, Theorem~\ref{thm1} applies also to Banach spaces that are isomorphic to a Hilbert space. However, it remains unclear whether such retractions exist for general Banach spaces, beyond the trivial case $X(2)\to X(1)$ given by the midpoint map $\{a,b\}\mapsto \{(a+b)/2\}$. 

Of particular interest here is the case $X=\ell^\infty$, because $\ell^\infty$ is an absolute $1$-Lipschitz retract, i.e., admits a $1$-Lipschitz retraction from any larger metric space containing it. Indeed, it remains unknown whether the property of being an absolute Lipschitz retract is inherited by finite subset spaces in general. See~\cite{BU} for the topological version of this problem, and ~\cite{Go}, \cite{AIMPL} for the Lipschitz version. 

Another setting to which Theorem~\ref{thm1} could be conceivably extended is CAT(0) metric spaces. The existence of $1$-Lipschitz retraction $X(2)\to X(1)$ for such spaces is a well-known consequence of the convexity of the metric in CAT(0) spaces (e.g.,~\cite{BH}): the map sending each pair of points to the midpoint of the geodesic connecting them provides such a retraction.  

\section{Proof of Theorem~\ref{thm1}}

Let $\HH^n$ be the Cartesian power of $\HH$, equipped with the metric 
\[
d((x_1,\dots,x_n),(y_1,\dots,y_n))=\left(\sum_{k=1}^n \|x_k-y_k\|^2 \right)^{1/2} 
\]
The product $\HH^n$  is also a Hilbert space. 
Define a function $\Phi \colon \HH^n\to \mathbb R$ by 
\begin{equation*}
\Phi(x_1,\dots,x_n)= \sum_{1\le i<j\le n} \|x_i-x_j\|
\end{equation*}
It is easy to see that $\Phi$ is a convex function on $\HH^n$. 
Let $D=\{x\in \HH^n \colon x_i=x_j \text{ for some } i\ne j\}$. The function $\Phi$ is Fr\'echet differentiable at every point of $\HH^n\setminus D$, with the derivative 
\begin{equation}\label{e2grad}
\nabla \Phi(x) =\left( \sum_{j\ne i}\frac{x_i-x_j}{\|x_j-x_i\|} \right)_{i=1}^n
\end{equation}
By~\eqref{e2grad}, $\Phi$ satisfies the upper gradient bound 
\begin{equation}\label{uppergrad}
\|\nabla \Phi\| \le (n-1)\sqrt{n}, \quad x\in \HH^n\setminus D
\end{equation}

Given a set in $\HH(n)\setminus \HH(n-1)$, enumerate its elements as $\{x_1,\dots,x_n\}$ (in arbitrary
order), thus associating to it a point  $x\in \HH^n\setminus D$. Since $x$ uniquely identifies the set $\{x_i\}$, we sometimes write $x$ instead of $\{x_i\}$ to simplify notation.  

Consider the ODE system
\be{ode}
\frac{du_i}{dt} = \sum_{j\ne i}\frac{u_j-u_i}{\|u_j-u_i\|},\quad i=1,\dots,n
\ee
with the initial conditions $u_i(0)=x_i$. In view of~\eqref{e2grad}, the system~\eqref{ode} can be seen as the gradient flow of the function $\Phi$. Note that the right hand side of~\eqref{ode} belongs to the finite-dimensional subspace spanned by $x_1,\dots,x_n$. Hence, the solution remains in this subspace as long as it exists. By the   Picard existence and uniqueness theorem, there is a unique solution  until $u$ reaches the set $D$.

Let $[0,T(x))$ be the maximal interval of existence of solution of~\eqref{ode}.  Denote $\delta(x) = \min_{i<j}\|x_i-x_j\|$. Since 
\be{lower}
\left\|\frac{du_i}{dt}\right\|  \le n-1 \quad \text{for all }i
\ee
it follows that 
\[
T(x)\ge \frac{\delta(x)}{2(n-1)}
\]

The following inequality provides an estimate for $T(x)$ in the reverse direction; it turns out that $T(x)$ is comparable to $\delta(x)$. 

\be{quick} 
T(x) \le \frac{\delta(x)}{2} 
\ee

\begin{proof}[Proof of~\eqref{quick}]
A map $F\colon\HH\to\HH$ is called monotone if 
\[\langle F(a)-F(b),a-b\rangle \ge 0 \quad \text{ for all } a,b\in\R^d\] 
It is a well-known fact~\cite{Ro}*{\S 24} that the gradient of any convex function is monotone. In particular, $F(x)=x/\|x\|$ is a monotone map, being the gradient of convex function $x\mapsto \|x\|$. 

Renumbering the points $x_i$, we may assume $\|x_1-x_2\|=\delta(x)$. Consider the function $\varphi(t)=\|u_1(t)-u_2(t)\|$, $0<t<t_c$. Differentiation yields 
\[\varphi'(t)=\|u_1-u_2\|^{-1}\left \langle \frac{du_1}{dt}-\frac{du_2}{dt} , u_1-u_2\right \rangle\] 
The inner product on the right consists of the term 
\[
\langle F(u_2-u_1)-F(u_1-u_2), u_1-u_2\rangle = -2\|u_1-u_2\|
\]
and the sum over $j=3,\dots,n$ of  
\[\begin{split} &\langle F(u_j-u_1)-F(u_j-u_2), u_1-u_2\rangle \\
&= - \langle F(u_j-u_1)-F(u_j-u_2), (u_j-u_1)-(u_j-u_2)\rangle \\ & \le 0. 
\end{split} \]
Thus, $\varphi'(t) \le -2$ for $0<t<T(x)$, and since $\varphi(t)\ge 0$ by definition, it follows that $T(x)\le \varphi(0)/2 = \delta(x)/2$. 
\end{proof}

We are now ready to define the retraction $r\colon X(n)\to X(n-1)$. On the subset $X(n-1)\subset X(n)$ it is the identity map. For a set $x=\{x_1,\dots,x_n\}\in X(n)\setminus X(n-1)$   
let $r(\{x_i\})=\{u_i(T(x))\}$. This is well-defined because a different enumeration of the elements 
$\{x_1,\dots,x_n\}$ would only result in a different enumeration of the elements $\{u_i(T(x))\}$. 

It remains to prove that $r$ is a Lipschitz 
retraction of $\HH(n)$ onto $\HH(n-1)$ in the Hausdorff metric $d_H$. Specifically,
\be{Lipgoal}
d_H(r(x),r(y))\le \max(n^{3/2},2n-1) \,d_H(x,y)
\ee
for all $x,y\in\HH(n)$. 

\begin{proof}[Proof of ~\eqref{Lipgoal}]
Let $(u_i)$ and $(v_i)$ be the solutions of~\eqref{ode}  with initial data $(x_i)$ and $(y_i)$, respectively.  

Combining ~\eqref{quick} and~\eqref{lower}, we obtain that 
\be{case1}
d_H(r(x),x) \le \frac{n-1}{2} \delta(x)
\ee
and similarly for $y$. 

\begin{lemma}\label{stable}
 $\sum_{i=1}^n \|u_i(t)-v_i(t)\|^2$ is a  nonincreasing function of $t$ for 
 $0<t< \min(T(x),T(y))$. 
\end{lemma}

\begin{proof} The point $(u_1(t),\dots,u_n(t))\in \HH^n $ evolves under the gradient flow of the convex function $\Phi(u_1,\dots,u_n) = \sum_{i<j}\|u_i-u_j\|$. Since the gradient of a convex function is monotone, we have  
\[\left\langle \frac{du}{dt}-\frac{dv}{dt}, u-v \right \rangle\le 0\] 
The left hand side is $1/2$ of the derivative of $\|u(t)-v(t)\|^2$ with respect to $t$, which proves the claim.
 \end{proof} 

As a consequence of Lemma~\ref{stable},
\be{stableH}
d_H(\{u_i\},\{v_i\}) \le \sqrt{n} \max_{i}\|x_i-y_i\|   
\ee
for all $0<t<\min(T(x),T(y))$. Let $\rho=d_H(x,y)$. 

\textbf{Case 1}: $\delta(x)+\delta(y)\le 4\,\rho$. Then from~\eqref{case1} 
we get 
\[ 
d_H(r(x),r(y)) \le \rho+ d_H(r(x),x)+d_H(r(y),y)  \le \rho + 2(n-1)\rho
\]
which implies~\eqref{Lipgoal}. 

\textbf{Case 2}: $\delta(x)+\delta(y)>  4\rho$. We may assume
$\delta(x)> 2\rho$. Since the function $\delta$ is $2$-Lipschitz in the Hausdorff metric, it follows that $\delta(y)>0$. 

The geometric meaning of $\delta(x)> 2\rho$ is that the points $x_i$ are separated by more than $2\rho$, yet each of them is within $\rho$ of some point $y_j$. Therefore, we can enumerate the points $x_i$ and $y_i$ in such a way that 
 \be{close}
 \|x_i-y_i\|\le \rho \quad\text{ for } \ i=1,\dots,n
 \ee
From now on we use only~\eqref{close}, in which the roles of $x$ and $y$  can be 
interchanged. Thus, we may assume that $T(x)\le T(y)$. 

By definition, $r(x) = \{u_i(T(x))\}$. Let $z =  \{v_i(T(x))\}$. By~\eqref{stableH} we have 
\be{goal1}
d_H(r(x),z)\le \sqrt{n}\,\rho 
\ee 
Since $\delta$ is $2$-Lipschitz and $\delta(r(x))=0$, it follows that 
\[\delta(z)\le 2d_H(r(x),z)\le 2\sqrt{n}\,\rho\] 
The estimate~\eqref{case1} yields 
\be{goal2}
d_H(r(z),z)\le (n-1)\sqrt{n}\,\rho
\ee
Now~\eqref{Lipgoal} follows from~\eqref{goal1} and~\eqref{goal2}: 
\[
d_H(r(x),r(y))\le d_H(r(x),z)+d_H(r(z),z)  \le n^{3/2}\rho
\qedhere \] 
\end{proof}

\section{An example and open questions}
 
Since the midpoint map $\HH(2)\to\HH(1)$ is Lipschitz with constant $1$, it is natural to ask whether a $1$-Lipschitz retraction of $\HH(n)$ onto $\HH(n-1)$ exists  for $n\ge 3$. The following example, given by the referee of an earlier version of this paper, shows that the answer is negative already for  $n=3$. 

\begin{example} There is no $1$-Lipschitz retraction from $\R^2(3)$ onto $\R^2(2)$. 
\end{example}

\begin{proof}
Let $A=\{(0,0),(1,0),(1/2,\sqrt{3}/2)\}$ be the set of vertices of an equilateral triangle of sidelength $1$ in the plane $\R^2$. Also let $B = \{(-1,0),(0,0)\}$ and $C=\{(1,0),(2,0)\}$; these sets lie on the line extending the base of the triangle. Then 
$
d_H(A,B)=d_H(A,C)=1
$
and $d_H(B,C)=2$. If there was a $1$-Lipschitz retraction of $\R^2(3)$ onto $\R^2(2)$, the image of $A$ would be some set $E\in \R^2(2)$ such that 
$d_H(E,B)\le 1$ and $d_H(E,C)\le 1$. The only such set is $\{(0,0),(1,0)\}$, formed by the vertices of the base of the triangle $A$. However, the above argument   also  applies to two other sides of $A$, which yields  a contradiction. 
\end{proof}

\begin{question}\label{q1}
Do there exist   retractions $\HH(n)\to  \HH(n-1)$ with the Lipschitz constants bounded independently of $n$? 
\end{question}

In conclusion we state  the questions mentioned in the introduction. 

\begin{question} \label{q3}
If $X$ is a CAT(0) metric space, do there exist Lipschitz retractions $X(n)\to  X(n-1)$ for every $n\ge 2$? 
\end{question}

\begin{question} \label{q2}
If $X$ is a Banach space, do there exist Lipschitz retractions $X(n)\to  X(n-1)$ for every $n\ge 2$? 
\end{question}

Although the linear span of every $n$-subset of a Banach space $X$ can be given an equivalent inner product metric (thus allowing for a Lipschitz retraction within this subspace), the retraction depends on the choice of renorming. Thus, it seems that Theorem~\ref{thm1} cannot be used to answer Question~\ref{q2}.

\end{document}